\documentclass[sn-mathphys-num]{sn-jnl}

\usepackage{graphicx}%
\usepackage{multirow}%
\usepackage{amsmath,amssymb,amsfonts}%
\usepackage{amsthm}%
\usepackage{mathrsfs}%
\usepackage[title]{appendix}%
\usepackage{xcolor}%
\usepackage{textcomp}%
\usepackage{manyfoot}%
\usepackage{booktabs}%
\usepackage{algorithm}%
\usepackage{algorithmicx}%
\usepackage{algpseudocode}%
\usepackage{listings}%
\usepackage{titlesec}
\usepackage{url}
\usepackage{ragged2e}

\theoremstyle{thmstyleone}%
\newtheorem{theorem}{Theorem}
\newtheorem{proposition}{Proposition}%

\theoremstyle{thmstyletwo}%
\newtheorem{example}{Example}%
\newtheorem{remark}{Remark}%
\newtheorem{assumption}{Assumption}%
\newtheorem{lemma}{Lemma}%
\theoremstyle{thmstylethree}%

\begin{document}

\title[Article Title]{A High-order Backpropagation Algorithm for Neural Stochastic Differential Equation Model}


\author{\fnm{Daili} \sur{Sheng}}\email{25b312008@stu.hit.edu.cn}

\author*{\fnm{Minghui} \sur{Song}\textsuperscript{*}}\email{songmh@hit.edu.cn}

\author{\fnm{Xiang} \sur{Peng}}\email{1201200313@stu.hit.edu.cn}
\author{\fnm{Xuanqi} \sur{Dong}}\email{23s012013@stu.hit.edu.cn}

\affil{\orgdiv{School of Mathematics}, \orgname{Harbin Institute of Technology}, \orgaddress{\city{Harbin}, \postcode{150001}, \country{China}}}

%


\abstract{Neural stochastic differential equation model with a Brownian motion term can capture epistemic uncertainty of deep neural network from the perspective of a dynamical system. {The goal of this paper is to improve the convergence rate of the sample-wise backpropagation algorithm  in neural stochastic differential equation model which has been proposed in [Archibald et al., SIAM Journal on Numerical Analysis, 62 (2024), pp. 593-621]. It is necessary to emphasize that, improving the convergence order of the algorithm consisting of forward backward stochastic differential equations remains challenging, due to the loss of information of Z term in backward equations under sample-wise approximation and the limitations of the forward network form.} In this paper, we develop a high-order backpropagation algorithm to improve the training accuracy. Under the convexity assumption, the result indicates that the first-order convergence is achieved when the number of training steps is proportional to the cubic number of layers. Finally, numerical examples illustrate our theoretical results.
}

\keywords{Neural stochastic differential equation model; Stochastic gradient descent; Convergence analysis.}

\maketitle
\section{Introduction}\label{sec1}
Neural stochastic differential equation (Neural SDE) model, also known as stochastic neural network (SNN), models the evolution of hidden states through a stochastic differential equation (SDE) \cite{jia2019neural, kidger2021efficient, kong2020sde, liu2020does}. This enables explicit uncertainty quantification, which is critical for decision making to avoid dangerous accidents in safety-critical areas, ranging from automatic medical diagnosis to autonomous vehicles to cyber security  and beyond \cite{guo2017calibration}. Some empirical results indicate Neural SDE has better uncertainty estimation  in some fields than classical methods, such as Bayesian Neural Networks (BNNs) \cite{kwon2020uncertainty, eaton2018towards}, hypothetical density filtering methods \cite{lakshminarayanan2017simple}, Monte Carlo methods \cite{gal2016dropout}, etc. Another study \cite{liu2020does} demonstrates that Neural SDE can have better robustness and generalization than Neural ordinary differential equation (Neural ODE) model \cite{chen2018neural}.

While the construction and justification of Neural SDE are well accepted, the training process is also challenging. Due to the stochastic integrals, the standard chain rule is not applicable for backpropagation like deterministic deep neural networks (DNNs), and It\^{o} calculus is needed, which makes the computation of the gradient complicated \cite{chen2018neural, li2020scalable}. To derive a mathematical expression for the gradient, recent work formulates Neural SDE training as a stochastic optimal control problem (SOCP), and applies the stochastic maximum principle (SMP) to solve it \cite{archibald2020stochastic, MR4470545, archibald2024numerical}, which leads to a stochastic Hamiltonian system that consists of forward backward stochastic differential equations (FBSDEs) \cite{ma1999forward, zhang2017backward}. Thus, solving SOCPs requires solving FBSDEs that satisfy specific optimization condition typically achieved by gradient descent. Under appropriate assumptions, it can be shown that the gradient process of the optimization condition can also be represented by a FBSDE system. Therefore, solving FBSDEs has to be implemented repeatedly to reach the optimization condition.

Several numerical schemes for solving FBSDEs have been developed, among  which some are Euler-type methods with convergence rate $1/2$, such as \cite{cvitanic2005steepest, delarue2006forward, douglas1996numerical} and some  are high-order numerical methods, such as \cite{ma2008numerical, zhao2014new, zhao2014second, zhao2014numerical}. However, approximating solutions of FBSDEs in the high-dimensional (controlled) state space at each iteration step remains a challenge. To approximate the gradient with classical numerical schemes, the conditional expectation has to be evaluated which is a very challenging task because of high-dimensional integrations. To address this challenge, we utilize the sample-wise approximation to drop the conditional expectation \cite{archibald2020stochastic, archibald2020efficient, archibald2023stochastic}. That is to say, we only select single sample-path in the state space and solve the FBSDEs along the chosen sample-path at each stochastic gradient descent (SGD) iteration step. In this way, we avoid solving the FBSDEs repeatedly in the entire high-dimensional state space, which makes the SGD optimization an efficient method to apply the SMP approach for Neural SDE. 

In \cite{archibald2024numerical}, following a standard algorithm flow as above, the convergence and an error estimate for the sample-wise backpropagation algorithm was proved. However, only half-order convergence was derived under the convexity assumption. {Due to the loss of information of Z term in backward equations under sample-wise approximation and the limitations of the forward network form, improving the convergence order of the algorithm remains challenging.} In this paper, we implement an efficient numerical scheme proposed in \cite{zhao2014second} as a basis for our high-order sample-wise backpropagation algorithm, and the network parameters can achieve first-order convergence. Under the convexity assumption, we prove that the error estimate contains two terms: the first is a quotient between the depth of neural networks and the number of iterations; the second is a first-order term with respect to the depth of neural networks, which is a half-order term in previous work \cite{archibald2024numerical}. While the first term reveals an inherent relation between the depth of a neural network and the number of training steps, the second term indicates the error of discretizing the continuous differential equations of probabilistic learning. In particular, by choosing the number of iterations through the depth of Neural SDE, the control variable can achieve first-order convergence.

The rest of this paper is organized as follows: In Section 2, we recall some known procedure of sample-wise backpropagation algorithm for Neural SDE and introduce our high-order scheme. The main convergence results are then stated and proved in Section 3. In Section 4, we validate our analysis results through several numerical experiments.

\section{\justifying A high-order sample-wise backpropagation method for Neural SDE}
 For the convenience of the readers, in this section, we detailedly introduce the high-order sample-wise backpropagation method for training Neural SDE based on the method proposed in \cite{MR4470545}. The core of our idea is to view the transformations as state evolution of a stochastic dynamical system. And then the  training procedure is a stochastic optimal control problem which can be solved by generalized stochastic gradient descent algorithm.

\subsection{Neural SDE and stochastic optimal control}
In a multilayer stochastic neural
network, the mathematical expression of sequential propagation between adjacent hidden layers can be described as
follows:
\begin{equation}\label{1001}
	X_{n + 1} = X_n+h b(X_n,u_n)+\sigma(u_n)\omega_n, \quad n = 0,1,2,\ldots,N - 1,
\end{equation}
in which $b$ and $\sigma$ act as the drift net and diffusion net for prediction and uncertainty quantification, respectively. Moreover, $X_n \in \mathbb{R}^p$ is the hidden state at layer $n$ containing $p$ neurons, $u_n$ denotes the parameters of the Neural SDE, $h$ is the step-size, and $\omega_n$ is a $q$-dimensional Gaussian random variable that accounts for uncertainty in the neural network. As neural nets map an input $X_0$ to an output $X_T$ through a sequence of hidden layers, the transformations can thus be viewed as the discretization of a dynamical system when $h \rightarrow 0$:
\begin{equation}\label{1002}
X_T = X_0 + \int_0^T b(X_t, u_t) \mathrm{d} t + \int_0^T \sigma(u_t) \mathrm{d} W_t, 
\end{equation}
where $\{ W_t^i\}_{ 0\leq t\leq T}, i=1,2,\dots,q$ is the standard Brownian motion corresponding to the i.i.d.  Gaussian random variable sequence $\{ w_n\}_n$ in \eqref{1001}. We propose the following objective function for training our Neural SDE:
\begin{equation}\label{1003}
	J(u) = \mathbb{E} \left[ \int_{0}^{T} r(X_t, u_t) \mathrm{d} t + \Phi (X_T, \Gamma) \right],
\end{equation}
where $\Gamma$ is the random variable that generates training data in machine learning,  which also depends on $X_0$, and $\Phi (X_T, \Gamma) :={\| X_T - \Gamma \|}_{loss} $ is the loss function corresponding to a loss error norm $\|  \cdot  \|_{loss}$ \cite{haber2017stable}, the integral $\int_{0}^{T} r(X_t, u_t) \mathrm{d} t$ represents the running cost. The goal of deep learning is to solve the SOCP, i.e., 
\begin{equation}\label{1004}
	\text{Find } u^{*} \in \mathcal{K}[0, T] \text{ such that } J(u^*) = \inf_{u \in \mathcal{K}[0, T]} J(u).
\end{equation}
The admissible control set is given by
$$\mathcal{K} [0,T] := \{ u \in L^2([0, T]; \mathbb{R}^m) | \, u(t) \in \mathcal{C} \text{ a.e.} \},$$
here $L^2([0, T]; \mathbb{R}^m)$ denotes the space consisting of all functions $u : [0, T] \rightarrow \mathbb{R}^m$ that satisfy $\|u\|_2^2 := \int_0^T |u(t)|^2 \, \mathrm{d}t < +\infty$ and $\mathcal{C} \subset \mathbb{R}^m$ is a nonempty, convex and closed subset.

\subsection{Stochastic gradient decent}
For the sake of notational simplicity, our discussion will be confined to the one-dimensional case, i.e., $p=m=q=1$, however, the entire framework can be trivially extended to the multi-dimensional case.

We begin with the following notation:
\begin{itemize}
	\item $C_b^{j,j,j}$: the set of continuously differentiable functions $(x, y, z) \in \mathbb{R} \times \mathbb{R} \times \mathbb{R} \times [0, T] \mapsto g(x, y, z) \in \mathbb{R}$ with bounded partial derivative functions $g_{x, y, z}^{j_1, j_2, j_3}$ for $0 \leq j_1, j_2, j_3 \leq j$. Analogous definitions apply for $C_b^j$, $C_b^{j,j}$.	
	\item $C_b^{j+\alpha}$: the set consisting of all $g \in C_b^j$ with $g^j$ being H{\"o}lder continuous with index $\alpha \in (0, 1)$.
	\item $\langle u, v\rangle$: the inner product of $u, v \in L^2([0, T]; \mathbb{R})$, i.e., $\langle u, v\rangle = \int_0^T u(t)  \cdot v(t) \, \mathrm{d} t$.
	\item $C$: the generic constant independent of $k, N$, and control parameter $u$.
\end{itemize}

For the functions in SOCP \eqref{1004}, the following assumptions are given throughout the paper:
\begin{assumption}\label{ass}
	\item[(a)] Both $b$ and $\sigma$ are deterministic, and $b \in C_b^{2,2}(\mathbb{R} \times \mathbb{R}; \mathbb{R})$ and $\sigma \in C_b^2(\mathbb{R}; \mathbb{R})$.
	\item[(b)] $b, b_x, b_u, \sigma, r_x, r_u$ are all uniformly Lipschitz in $x, u$ and uniformly bounded.
	\item[(c)] $\sigma$ satisfies the uniform elliptic condition.
	\item[(d)] The initial condition $X_0 \in L^2(\mathcal{F}_0)$.
	\item[(e)] The terminal (loss) function $\Phi \in C_b^{3+\alpha}$ for some $\alpha \in (0, 1)$.
	\item[(f)] $\lim_{\|u\|_2 \to \infty} J(u) = \infty$.
\end{assumption}
Notice that under Assumption \ref{ass}, the solution $X_t$ of \eqref{1002} and the cost functional $J(u)$ are all well defined for $u \in \mathcal{K}[0, T]$ (see \cite{yong2012stochastic}).

To utilize the stochastic gradient decent, we need to derive $\nabla J_u (u_t)$ first, which can be represented by introducing a BSDE. From the definition in \eqref{1003}, for any \( v \in L^2([0, T]; \mathbb{R}^m) \), we have
\begin{equation}\label{1005}
	\begin{aligned}
		&\nabla J_u (u_t)(v) = \lim_{\kappa \to 0} \frac{J(u + \kappa v) - J(u)}{\kappa}\\
		=& \mathbb{E} \left[ \int_0^T \Big(r_x (X_t, u_t) D X_t(v) + r_u (X_t, u_t) v(t) \Big) \mathrm{d}t + \Phi_x (X_T, \Gamma) D X_T(v) \right],
	\end{aligned}
\end{equation}
where $ t \mapsto D X_t(v)$ is the variational process given by the following SDE:
\begin{equation}\label{10061}
	\begin{aligned}
		\mathrm{d} D X_t(v) &= \left( b_x(X_t(v), u_t) D X_t(v) + b_u(X_t(v), u_t) v(t) \right) \mathrm{d} t \\
		&\quad + \left( \sigma_x(X_t(v), u_t) D X_t(v) + \sigma_u(u_t) v(t) \right) \mathrm{d} W_t, \quad D X_0(v) = 0,
	\end{aligned}
\end{equation}
and $b_x, \sigma_x$ and $r_x$ are partial derivatives with respect to the state $X$, $b_u, \sigma_u$ and $r_u$ are partial derivatives with respect to the control $u$. To get rid of $DX_t(v)$ in \eqref{1005}, 
we have the following BSDE:
\begin{equation}\label{1006}
	-\mathrm{d} Y_t = \left( b_x(X_t, u_t) Y_t + r_x(X_t, u_t) \right) \mathrm{d}t - Z_{t} \mathrm{d} W_t, \quad Y_T = \Phi_x(X_T, \varGamma),
\end{equation}
in which $Y_t$ is the adjoint process of the state $X_t$, and $Z_t$ is the martingale representation of $Y_t$ with respect to $W_t$.
Then under Assumption 1, the BSDE \eqref{1006} admits an unique solution $(Y_t, Z_t)$ for $u \in \mathcal{K}$, and the following boundness property is guaranteed (see Theorem 4.2.1 in \cite{zhang2017backward}):
\begin{equation}
	\sup_{0 \leq t \leq T} \mathbb{E}[|Y_t|^2] + \mathbb{E} \left[ \int_0^T |Z_t|^2 \, \mathrm{d} t \right] \leq C.
\end{equation}
We shall show in the following that by introducing the pair $(Y_t, Z_t)$, the involving  terms $D X_t (v)$ in \eqref{1005} will be canceled. More precisely, by It\^{o} formula, we have
\begin{equation}\label{10062}
	\begin{aligned}
		& r_x (X_t, u_t) D X_t(v) \mathrm{d}t \\
		=& -D X_t(v) \mathrm{d} Y_t -  Y_t b_x(X_t, u_t)   D X_t(v) \mathrm{d} t + Z_t D X_t(v) \mathrm{d} W_t \\
		=& -\mathrm{d}(Y_t D X_t(v)) + Y_t \mathrm{d} D X_t(v) + Z_t  \sigma_u(u_t) v(t)  \mathrm{d} t \\
		& - \left( Y_t b_x(X_t, u_t)  \right) D X_t(v) \mathrm{d} t + Z_t D X_t(v) \mathrm{d} W_t \\
		=& -\mathrm{d} (Y_t D X_t(v)) + \left( Y_t b_u(X_t, u_t) + Z_t \sigma_u(u_t) \right) v(t) \mathrm{d} t \\
		& + \left( Y_t \sigma_u(u_t) v(t) + Z_t D X_t(v) \right) \mathrm{d} W_t.
	\end{aligned}
\end{equation}
Then, by inserting \eqref{10062} into \eqref{1005}, we can re-define $\nabla J_u$ by
\begin{equation}
	\nabla J_u (u_t) = \mathbb{E} \left[ b_u(X_t, u_t) Y_t + \sigma_u(u_t) Z_t + r_u(X_t, u_t) \right].
\end{equation}

We close this section by the following lemmas of projection operator and stochastic gradient decent method.
\begin{lemma}\label{lemma01}
	Let $\mathcal{P}_{\mathcal{K}}$ be the projection operator from $L^2([0, T]; \mathbb{R}^m)$ onto a convex set $\mathcal{K}$ such that
	\begin{equation}\label{le01}
		\|v - \mathcal{P}_{\mathcal{K}} v\| = \min_{z(t) \in \mathcal{K}} \|v - z\|.
	\end{equation}
	Then $\mathcal{P}_{\mathcal{K}}v$ satisfies \eqref{le01} if and only if, for any $z(t) \in \mathcal{K}$
	\begin{equation}
		(\mathcal{P}_{\mathcal{K}}v - v, z - \mathcal{P}_{\mathcal{K}}v) \geq 0.
	\end{equation}
\end{lemma}
For the SOCP, it is well known that for the optimal control $u^*$ it holds
\begin{equation}\label{d1}
	\left(\nabla J_u(u^*), v - u^*\right) \geq 0,
	\notag
\end{equation}
which implies that
\begin{equation}\label{d2}
	\left(u^* - \left(u^* - \eta \nabla J_u(u^*)\right), v - u^*\right) \geq 0,
\end{equation}
where $\eta$ is a positive constant. From Lemma \ref{lemma01}, \eqref{d2} implies that
\begin{equation}\label{d3}
	u^* = \mathcal{P}_{\mathcal{K}}(u^* - \eta \nabla J_u(u^*)).
\end{equation}
That is, the optimal control $u^
*$ is the fixed point of $P_{\mathcal{K}} 
(u - \eta \nabla J(u))$ on $\mathcal{K}$. 
The following lemmas state that the projection operators are nonexpansive in the $L_2$-norm,
which form the theoretical foundation for our 
later proofs.
\begin{lemma}
	For the projection $\mathcal{P}_{
		\mathcal{K}}$, it holds that
	\[
	\|\mathcal{P}_{
		\mathcal{K}}w - \mathcal{P}_{
		\mathcal{K}}z\|_2 \leq \|w - z\|_2,
	\]
	for any $w, z \in L^2([0, T]; \mathbb{R}^m)$.
\end{lemma}
\begin{proof}
	Using Lemma \ref{lemma01} and Cauchy-Schwarz inequality yields the result. 
\end{proof}

Having the gradient of $J$ and projection operator $\mathcal{P}_{\mathcal{K}}$ in hand, we can carry out gradient descent optimization to determine the optimal control as follows:
\begin{equation}\label{1007}
	u^{k + 1}_t =  \mathcal{P}_{\mathcal{K}} \left( u^k_t-\eta_k\nabla J_u (u^k_t) \right), \quad k = 0,1,2,\ldots, \quad 0\leq t\leq T,
\end{equation}
where $u^0$ is an initial guess for the optimal control, $\eta_k$ is the stepsize of gradient descent in the $k$th iteration step. For the stochastic gradient descent method introduced in \cite{MR4470545}, we can choose one sample of $X_t$ and $Z_t$ and modify \eqref{1007} as follows:
\begin{equation}\label{1008}
	\begin{aligned}
	u^{k+1}_t &= \mathcal{P}_{\mathcal{K}} \left( u^k_t-\eta_k\left[b_u\left({X}^k_t,u^k_t\right){Y}^k_t + \sigma_u\left(u^k_t\right){Z}^k_t+r_u\left({X}^k_t,u^k_t\right)\right] \right),\\
	k &= 0,1,2,\ldots, \quad 0\leq t\leq T.
	\end{aligned}
\end{equation}

\subsection{Temporal discretization for optimal control}
The optimal control $u^*$ is approximated by step function. A uniform time partition $\Pi_N = \{t_0, \ldots, t_N\}$ over $[0, T]$ is introduced:
\[
0 = t_0 < t_1 < \ldots < t_N = T, \; h = t_{n+1} - t_n = T/N, 
\]
where $N$ is the partition number, which is equivalent to the depth of stochastic neural networks. We define the associated space of piecewise constant functions by
$$\mathcal{U}_N [0, T]= \left\{ u \in  L^2([0, T]; \mathbb{R}^m) \mid u = \sum_{n=1}^{N} \alpha_n 1_{[t_n, t_{n+1})} \text{ a.e., } \alpha_n \in \mathbb{R}^m \right\}.$$
Let $\mathcal{K}_N[0, T] = \mathcal{K}[0, T] \cap \mathcal{U}_N[0, T]$, the approximated problem of \eqref{1004} is given by 
\begin{equation}\label{10044}
	\text{Find } u^{*,N} \in \mathcal{K}_N[0, T] \text{ such that } J(u^{*,N}) = \inf_{u \in \mathcal{K}_N[0, T]} J(u).
\end{equation}
Numerical implementation of the gradient descent scheme \eqref{1007} requires numerical approximations to the SDE \eqref{1002} and BSDE \eqref{1006}. Hence, we need to solve (for $t \in [0, T ]$) the following FBSDEs:
\begin{equation}
	\left\{
	\begin{aligned}
		\mathrm{d} X_t &= b(X_t, u_t) \mathrm{d} t + \sigma(u_t) \mathrm{d} W_t, \quad\,\, X_{t=0} = X_0, \\
		-\mathrm{d} Y_t &= f(X_t, Y_t, u_t) \mathrm{d}t - Z_{t} \mathrm{d} W_t, \qquad\, Y_T = \Phi_x(X_T, \varGamma),
	\end{aligned}
	\right.
\end{equation}
where $f(X_t, Y_t, u_t) =  b_x(X_t, u_t) Y_t + r_x(X_t, u_t)$.
\begin{remark}
Let Assumption \ref{ass} hold, it is well known that the above backward equation is wellposed \cite{peng1991probabilistic}. Moreover, the solutions $Y_t$ and $Z_t$ have the representations
	\begin{equation}\label{20}
		Y_t = \eta(t, X_t), \qquad Z_t = \sigma(u(t)) \partial_x \eta(t, X_t),
	\end{equation}
	where $\eta(t, x) : [0, T] \times \mathbb{R}^p \rightarrow \mathbb{R}^q$ is the solution of the following parabolic PDE
	\begin{equation}
		\mathcal{L}^0 \eta(t, x) = -f\Big(x, \eta(t, x), u(t)\Big), \qquad \eta(T, x) = \Phi_x(x, \Gamma),
	\end{equation}
	with
	\begin{equation*}
		\mathcal{L}^0 \eta(t, x) = \partial_t \eta(t, x) + b(x, u(t)) \partial_x \eta(t, x) + \frac{1}{2} \sigma(u(t))^2 \partial_{xx} \eta(t, x).
	\end{equation*}
	The representation in \eqref{20} is the so called nonlinear Feynman-Kac formula \cite{peng1991probabilistic}. Furthermore, if $b \in C_b^4, f \in C_b^{4,4}$, and $\Phi_x \in C_b ^{4+\alpha}$ for some $\alpha \in (0, 1)$, then $\eta \in C_b^4$. 
\end{remark}
In order to obtain a high-order discretization method for control variable $u$ than \cite{archibald2024numerical}, an efficient numerical scheme in \cite{zhao2014numerical} for decoupled FBSDEs has been implemented: 
\begin{align}
	Y_n^N &= \mathbb{E}_{t_n}^{X_n^N} \left[ Y_{n+1}^N \right] + \frac{1}{2} h f_n^N + \frac{1}{2} h \mathbb{E}_{t_n}^{X_n^N} \left[ f_{n+1}^N \right], \label{b1}\\
	\frac{1}{2} h Z_n^N &= \mathbb{E}_{t_n}^{X_n^N} \left[ Y_{n+1}^N \Delta \tilde{W}_{t_{n+1}} \right] + h \mathbb{E}_{t_n}^{X_n^N} \left[ f_{n+1}^N \Delta \tilde{W}_{t_{n+1}} \right], \label{b2}
\end{align}
with
\begin{equation}\label{b3}
	X_{n + 1}^N = X_n^N + h b(X_n^N,u_{t_n})+\sigma(u_{t_n})\Delta W_{t_{n+1}}, \quad n = 0,1,2,\ldots,N - 1,
\end{equation}
where $X_{n+1}^N$, $Y_{n}^N$, $Z_{n}^N$, $f_{n+1}^N$ are the numerical approximation for $X_{t_{n+1}}^{t_n,X_n^N}$, $Y_{t_n}^{t_n,X_n^N}$, $Z_{t_n}^{t_n,X_n^N}$, $f_{t_{n+1}}^{t_n, X_n^N}$
respectively, and $\Delta \tilde{W}_s$ is defined by
\begin{equation}\label{a7}
	\Delta \tilde{W}_s = 2 \Delta W_s - \frac{3}{h} \int_{t_n}^s (r - t_n) \mathrm{d} W_r.
\end{equation}
The next lemma shows the convergence of the numerical solutions to BSDEs. 
\begin{lemma}
	Assume Assumption \ref{ass} holds, and $f \in C_b^{4,4}, b \in C_{b}^{4}, \Phi_x \in C_b^{4+\alpha}, \alpha \in (0, 1)$, $\mathbb{E}[|Y_{t_N}^{t_N, X_N^N} - Y_N^N|^2] \leq Ch^2$, $\mathbb{E}[|Z_{t_N}^{t_N, X_N^N} - Z_N^N|^2] \leq Ch^2$, then we obtain the error estimate of scheme \eqref{b1} \eqref{b2} as
	\begin{equation}
		\mathbb{E}\left[\|Y_{t_n}^{t_n,X_n^N} - Y_n^N\|_2^2\right] + h \sum_{i=n}^{N-1} \mathbb{E}\left[\|Z_{t_i}^{t_i,X_i^N} - Z_i^N\|_2^2\right] \leq Ch^2.
	\end{equation}
\end{lemma}
\begin{remark}
	Up to now, we have proposed the numerical schemes for BSDE \eqref{1006} to implement the gradient decent scheme \eqref{1007} numerically. However, in order to implement numerical schemes \eqref{b1}-\eqref{b3}, one needs to approximate the (conditional) expectations. One of well-known numerical methods for approximating expectations is Monte Carlo simulation, which requires high computational cost especially when the dimension of the controlled state $X_t$ is high and the discretization number $N$ is large. Thus, a natural problem is whether such an expectation can be removed in computation as we can randomly select one data sample in SGD and the controlled state process $X_t$ can be viewed as ``pseudo-data''. The detailed process and proof will be discussed later.
\end{remark}

To address the aforementioned computational challenges in Monte Carlo simulation, we introduce an enhanced sample-wise stochastic gradient descent algorithm to carry out the optimization procedure.
First, the sample-wise numerical solutions $X_n^k$ of $X$, and  $(Y_n^k, Z_n^k)$ of $(Y, Z)$, are given by
\begin{equation}\label{1014}
	\begin{aligned}
		X_{n+1}^k  =& X_{n}^k + h b(X_{n}^k,u_{t_n}^k)+\sigma(u_{t_n}^k) \omega_{n+1}^{k},\\
		Y^{k}_{n} =& Y_{n + 1}^{k} + \frac{1}{2} h \left[b_x (X_{n+1}^{k},u^k_{t_{n+1}}) Y_{n+1}^{k}+r_x(X_{n+1}^{k},u^k_{t_{n+1}})\right]\\
		&+ \frac{1}{2} h \left[b_x (X_{n}^{k},u^k_{t_{n}}) Y_{n}^{k}+r_x(X_{n}^{k},u^k_{t_{n}})\right],\\
		\frac{1}{2}hZ_{n}^{k} =& Y_{n+1}^{k} \tilde{\omega}_{n+1}^k + h \left[b_x (X_{n+1}^{k},u^k_{t_{n+1}}) Y_{n+1}^{k}+r_x(X_{n+1}^{k},u^k_{t_{n+1}})\right] \tilde{\omega}_{n+1}^k,
	\end{aligned}
\end{equation}
where $\omega_{n+1}^k$ and $\tilde{\omega}_{n+1}^k$ are the samples for $\Delta W_{t_{n+1}}$, $\Delta \tilde{W}_{t_{n+1}}$ respectively.
With \eqref{1014}, we approximate the gradient $\nabla J_{u}$ by
\begin{equation}\label{1015}
	\nabla j_{u}^{k}\left(u_{t_n}^{k}\right) := b_{u}\left(X_{n}^{k}, u_{t_n}^{k}\right) Y_{n}^{k} + \sigma_{u}\left(u_{t_n}^{k}\right) Z_{n}^{k} + r_{u}\left(X_{n}^{k}, u_{t_n}^{k}\right). 
\end{equation}
Let $\mathcal{P}_{\mathcal{K}_N}$ be the projection operator onto the ${\mathcal{K}_N}[0, T]$. We can show that
\begin{equation}\label{17}
	u^{*,N} = \mathcal{P}_{\mathcal{K}_N} \left( u^{*,N} - \eta \nabla J_u(u^{*,N}) \right),
\end{equation}

and the following lemma:
\begin{lemma}
	For the projection $\mathcal{P}_{
		\mathcal{K}_N}$, it holds that
	\[
	\|\mathcal{P}_{
		\mathcal{K}_N}w - \mathcal{P}_{
		\mathcal{K}_N}z\|_2 \leq \|w - z\|_2,
	\]
	for any $w, z \in L^2([0, T]; \mathbb{R}^m)$.
\end{lemma}
Then we can derive the sample-wise stochastic gradient descent (SGD) scheme as follows:
\begin{equation}\label{1016}
	u_{t_{n}}^{k+1} = \mathcal{P}_{\mathcal{K}_N} \left(u_{t_{n}}^{k} - \eta_{k} \nabla j_{u}^{k} (u_{t_{n}}^{k})\right), \quad k = 0, 1, 2, \ldots, \quad 0 \leq n \leq N.
\end{equation}
\begin{remark}
	Although a sample-wise backpropagation method based on Euler method has already been proposed in \cite{archibald2024numerical}, it is necessary to notice that the Euler method is a simple fixed-step numerical discrete method, and the approximate solution of calculus equations is limited, which also limits its convergence order. Our backpropagation method is proposed to enhance the numerical accuracy of 
	control variables in the network.
\end{remark}
We summarize the algorithm flow in Algorithm 1. To facilitate comparison with \cite{archibald2024numerical}, we highlight the part belonging to our scheme in red only.
\begin{algorithm}
	\caption{A high-order backpropagation method}
	\begin{algorithmic}[1]
		\State Formulate the Neural SDE \eqref{1002} as the stochastic optimal control problem \eqref{1004} and give a partition $\Pi^N$ to the control problem as the depth of stochastic neural network.
		\State Choose the number of SGD iteration steps $K \in \mathbb{N}$, the learning rate $\{\eta_k\}_k$ and the initial guess for the optimal control $\{u_{t_n}^0\}_n$.
		\For{SGD iteration steps $k = 0, 1, 2, \cdots, K-1$}
		\State Simulate one realization of the state process $\{X_n^k\}_n$ through the scheme \eqref{b3}.
		\State \textcolor{red}{Simulate $\{({Y}_n^k, {Z}_n^k)\}_n$ through the schemes \eqref{b1},  \eqref{b2}};
		\State Calculate the gradient process and update the estimated optimal control $\{u_{t_n}^{k+1}\}_n$ through the SGD iteration scheme \eqref{1016};
		\EndFor
		\State The estimated optimal control is given by $\{u_{t_n}^K\}_n$;
	\end{algorithmic}
\end{algorithm}
\section{Convergence analysis}\label{sec2}
In this section, we analyze the convergence of the SGD algorithm \eqref{1014}-\eqref{1016}. Under the assumption that the cost function for the optimal control is convex, we derive
the first-order convergence rate for our algorithm in the meansquare sense.
\subsection{Sample-wise numerical solution as an unbiased estimation}
To illustrate that the stochastic approximation $\nabla j_u^k(u_{t_n}^k)$ of above high-order sample-wise backpropagation method for Neural SDE is indeed an unbiased estimator of the gradient $\nabla J_u^N (u_{t_n}^k)$, we first introduce 
the fact that the sample-wise solutions $Y_n^k$ and $Z_n^k$ introduced in \eqref{1014} are equivalent to the classic  numerical solutions $Y_n^N$ and $Z_n^N$ introduced in \eqref{b1}-\eqref{b2} under conditional expectation $\mathbb{E}_{t_n}^{X_n^N}[\cdot]$. Specifically, we have the following proposition, which is also the foundation of the convergence analysis.
\begin{proposition}\label{pro1}
	For given estimated control $u^k \in \mathcal{K}_N$, let $Y_n^{k,N}$ and $Z_n^{k,N}$ be the numerical solutions defined in \eqref{b1} \eqref{b2}. Then the following identities hold:
	\begin{equation}
		\mathbb{E}_{t_n}^{X_n^N}[Y_n^{k}] = Y_n^{k,N} \big|_{X_n^N}, \quad \mathbb{E}_{t_n}^{X_n^N}[Z_n^{k}] = Z_n^{k,N} \big|_{X_n^N}, \quad 0 \leq n \leq N-1,
	\end{equation}
	and therefore we have $\mathbb{E}[Y_n^{k}] = \mathbb{E}[Y_n^{k,N}]$ and $\mathbb{E}[Z_n^{k}] = \mathbb{E}[Z_n^{k,N}]$.
\end{proposition}
\begin{proof}
	Observe that the random variable $\omega_n^k$ in the scheme \eqref{1014} follows the same distribution as $\Delta W_{t_n}$ appeared in \eqref{b3}. Consequently, $\omega_n^k$ and $\Delta W_{t_n}$ are equivalent under expectation. More generally, for any function $\phi(\{\omega_n^k\}_n)$ acting on the sample path $\{\omega_n^k\}_n$, the equality $\mathbb{E}[\phi(\{\omega_n^k\}_n)] = \mathbb{E}[\phi(\{\Delta W_{t_n}\}_n)]$ holds. Similarly, following the same argument, we also have $\mathbb{E}[\phi(\{\tilde{\omega}_n^k\}_n)] = \mathbb{E}[\phi(\{\Delta \tilde{W}_{t_n}\}_n)]$.
	
	The proof proceeds by first examining the case $n = N-1$  (i.e., take one step back from the terminal time), and let $h$ be sufficiently small, such that $ 1 - 	\frac{1}{2} h b_x (X_{N-1}^{k},u^k_{t_{N-1}}) \neq 0$, we have
	\begin{equation}
		\begin{aligned}
			\mathbb{E}_{t_{N-1}}^{X_{N-1}^N} \left[Y_{N-1}^k \right] =& \mathbb{E}_{t_{N-1}}^{X_{N-1}^N} \Bigg[ \left( 1 - \frac{1}{2} h b_x (X_{N-1}^{k},u^k_{t_{N-1}}) \right)^{-1} \\			
			& \Bigg(Y_N^k + \frac{1}{2} h \left[b_x (X_{N}^{k},u^k_{t_{N}}) Y_{N}^{k}+r_x(X_{N}^{k},u^k_{t_{N}})\right] + \frac{1}{2} h r_x(X_{N-1}^{k},u^k_{t_{N-1}})\Bigg) \Bigg].\\
		\end{aligned}
		\notag
	\end{equation}

	Since $Y_N^k = \Phi_x = Y_N^{k,N}$ and $\mathbb{E}_{t_{N-1}}^{X_{N-1}^N} [X_N^k] = \mathbb{E}_{t_{N-1}}^{X_{N-1}^N}[X_N^{k,N}]$, where $X_N^{k,N}$ is the approximated solution introduced in \eqref{b3} with the given control $u^k$, the above equation becomes
		\begin{align}\label{111}
			&\mathbb{E}_{t_{N-1}}^{X_{N-1}^N} \left[Y_{N-1}^k \right] \notag\\
			=& \mathbb{E}_{t_{N-1}}^{X_{N-1}^N} \Bigg[ \left( 1 - 	\frac{1}{2} h b_x (X_{N-1}^{k,N},u^k_{t_{N-1}}) \right)^{-1} \\			
			& \Bigg(Y_N^{k,N} + \frac{1}{2} h \left[b_x (X_{N}^{k,N},u^k_{t_{N}}) Y_{N}^{k,N}+r_x(X_{N}^{k,N},u^k_{t_{N}})\right] + \frac{1}{2} h r_x(X_{N-1}^{k,N},u^k_{t_{N-1}})\Bigg) \Bigg]\notag\\
			=& \mathbb{E}_{t_{N-1}}^{X_{N-1}^N} \Bigg[ \left( 1 - 	\frac{1}{2} h b_x (X_{N-1}^{k,N},u^k_{t_{N-1}})\right)^{-1} \left( 1 - 	\frac{1}{2} h b_x (X_{N-1}^{k,N},u^k_{t_{N-1}}) \right)Y_{N-1}^{k,N} \Bigg]\notag\\
			=& Y_{N-1}^{k,N} \big|_{X_{N-1}^N}.\notag
		\end{align}
	Following the same argument, we also have 
		\begin{align}\label{112}
			&\mathbb{E}_{t_{N-1}}^{X_{N-1}^N} \left[ Z_{N-1}^k \right] = \mathbb{E}_{t_{N-1}}^{X_{N-1}^N} \left[ \frac{2Y_N^k \tilde{\omega}_{N}^k}{h} + 2 \left[b_x (X_{N}^{k},u^k_{t_{N}}) Y_{N}^{k}+r_x(X_{N}^{k},u^k_{t_{N}})\right] \tilde{\omega}_{n+1}^k \right] \notag\\
			=&\mathbb{E}_{t_{N-1}}^{X_{N-1}^N} \left[ \frac{2Y_N^{k,N} \Delta \tilde{W}_{t_{N}}^k}{h} + 2 \left[b_x (X_{N}^{k,N},u^k_{t_{N}}) Y_{N}^{k,N}+r_x(X_{N}^{k,N},u^k_{t_{N}})\right] \Delta \tilde{W}_{t_{N}}^k \right]\\
			= &Z_{N-1}^{k,N} \big|_{X_{N-1}^N}.\notag
		\end{align}
	Then, by repeatedly applying the equality \eqref{111}, \eqref{112} and the tower property, we obtain the desired result.
\end{proof}
Next, we proceed to derive the unbiased property of the gradient of the cost functional $J$. The analysis begins by constructing an augmented $\sigma$-algebra $\mathcal{G}_k := \sigma(\omega^i, \gamma^i, 0 \leq I \leq k-1)$ generated by the Gaussian random variables $\omega^i$, which we use to
generate state sample path $X^k$ in the sample-wise scheme \eqref{1014}, and the data sample $\gamma^i$ generated by the training data $\Gamma$. As demonstrated in the above proposition, we see that the stochastic approximation $\nabla j_u^k(u_{t_n}^k)$ introduced in \eqref{1015} is an unbiased estimator for the gradient $\nabla J_u^N(u_{t_n}^k)$ given $\mathcal{G}_k$, i.e.
	\[
	\mathbb{E}[\nabla j_u^k(u_{t_n}^k) \mid \mathcal{G}_k] = \nabla J_u^N(u_{t_n}^k).
	\]
Denote $\mathbb{E}^k[\cdot] := \mathbb{E}[\cdot \mid \mathcal{G}_k]$ in the rest of this paper for convenience of presentation. The following lemma is about the boundedness of the sample-wise solution $Y_n^k$ and the linear growth property for $Z_n^k$ with any approximate control $u^k \in \mathcal{K}_N$.
\begin{lemma}\label{lemma1}
	Under Assumption \ref{ass} (a)--(e), there exists a constant $C > 0$, such that
	\[
	\sup_{0 \leq n \leq N} \mathbb{E}[(Y_n^k)^2] \leq C, \quad \sup_{0 \leq n \leq N} \mathbb{E}[(Z_n^k)^2] \leq CN.
	\]
\end{lemma}
\begin{proof}
	We square both sides of the scheme
	\begin{align}
		\left( 1 - \frac{1}{2} h b_x (X_{n}^{k},u^k_{t_{n}}) \right) Y_{n}^k   = & \left(
		1 + \frac{1}{2} h b_x (X_{n+1}^{k},u^k_{t_{n+1}}) \right)Y_{n+1}^k \notag\\
		& +\frac{1}{2} h r_x (X_{n+1}^{k},u^k_{t_{n+1}}) + \frac{1}{2} h r_x (X_{n}^{k},u^k_{t_{n}})\\
		& \ \left( 1 - \frac{1}{2} h b_x (X_{n}^{k},u^k_{t_{n}}) \right)^2 \left(Y_{n}^k \right)^2  \notag\\			
		\leq &  (1+h)\left(
		1 + \frac{1}{2} h b_x (X_{n+1}^{k},u^k_{t_{n+1}}) \right)^2 \left(Y_{n+1}^{k} \right)^2  \notag\\
		& + \left(1+ \frac{1}{h}\right)\left[  \frac{h^2}{4} {\left(r_x (X_{n+1}^{k},u^k_{t_{n+1}}) +  r_x (X_{n}^{k},u^k_{t_{n}}) \right)}^2 \right],\notag
	\end{align}
	let $h$ be sufficiently small, such that $ 1 - 	\frac{1}{2} h b_x (X_{n}^{k},u^k_{{n}}) \neq 0$, $ 0 \leq n \leq N$, take expectation to obtain
	\begin{equation}
		\begin{aligned}
			& \mathbb{E} \left[ \left(Y_{n}^k \right)^2 \right] \leq  \mathbb{E} \left[(1+h) \left( \frac{1 + \frac{1}{2} h b_x (X_{n+1}^{k},u^k_{t_{n+1}})}{1 - \frac{1}{2} h b_x (X_{n}^{k},u^k_{t_{n}})} \right)^2  \left( Y_{n+1}^k \right)^2 \right]+Ch\\
			\leq & (1+Ch)\mathbb{E} \left[\left(Y_{n+1}^k \right)^2  \right]+Ch.
		\end{aligned}
	\end{equation}
	Then, by the discrete Gronwall inequality, we have
	\begin{equation}
		\sup_{0 \leq n \leq N} \mathbb{E}[(Y_n^k)^2] \leq C.
	\end{equation}
	Following the same argument, we also have
	\begin{equation}
		\sup_{0 \leq n \leq N} \mathbb{E}[(Z_n^k)^2] \leq {CN}.
	\end{equation}
\end{proof}
As a consequence of the above discussions, we have the following lemma.
\begin{lemma}\label{lemma2}
	Under Assumption \ref{ass}, for any $ u^k \in \mathcal{K}_N$, the following estimation holds:
	\begin{equation}
		\mathbb{E}\left[\big\Vert \nabla j_u^k(u^k) - \nabla J_u^N(u^k)\big\Vert_2^2\right] \leq CN.
	\end{equation}
\end{lemma}
\begin{proof}
Due to Lemma \ref{lemma1} and the boundedness assumptions for $b_u$, $\sigma_u$, and $r_u$, we have that
\begin{equation}
	|\nabla j_u(u_{t_n}^k)|^2 \leq C\left(|Y_n^k|^2 + |Z_n^k|^2\right) \leq CN.
\end{equation}
Then we can obtain
\begin{align*}
	\mathbb{E}\left[\big\Vert \nabla j_u^k(u^k) - \nabla J_u^N(u^k)\big\Vert_2 ^2\right]
	&\leq 2 \mathbb{E}\left[\big\Vert \nabla j_u^k(u^k) \big\Vert_2 ^2 \right] + 2 \mathbb{E}\left[\big\Vert \nabla J_u^N(u^k)\big\Vert_2 ^2 \right] \\
	&\leq 2h \sum_{n=0}^{N-1} |\nabla j_u(u_{t_n}^k)|^2 + Ch \sum_{n=0}^{N-1} \mathbb{E}\Big[\big|b_u(X_{t_n}^{k,N}, u_{t_n}^k) Y_n^{k,N}\big|^2 \\
	&\quad + \big|\sigma_u( u_{t_n}^k) Z_n^{k,N}\big|^2 + \big|r_u(X_n^{k,N}, u_{t_n}^k) \big|^2\Big] \\
	&\leq CN + C h \sum_{n=0}^{N-1}\sup_{0 \leq n \leq N-1} \mathbb{E}\left[|Y_n^{k,N}|^2 + |Z_n^{k,N}|^2\right]+C \\
	&\leq CN+C,
\end{align*}
where $C > 0$ is a generic constant independent of $N$. Hence, we can get the desired result from the above analysis. 
\end{proof}
\subsection{Convergence analysis}
	In this subsection, we will turn to error estimates for $u^{K+1} - u^*$. To do this, we first introduce the following lemmas.
\begin{lemma}\label{lemma3}
	Assume that Assumption \ref{ass} holds and $f \in C_b^{4,4}, b \in C_{b}^{4}, \Phi_x \in C_b^{4+\alpha}$ for some $\alpha \in (0, 1)$, $u^k \in \mathcal{K}_N$. Then there exists a constant $C > 0$ such that
	\begin{equation}
		\sup_{u^k \in \mathcal{K}_N} \big\Vert\nabla J_{u}(u^{k})-\nabla J_{u}^{N}(u^{k})\big\Vert_2^{2} \leq \frac{C}{N^2}.
	\end{equation}
\end{lemma}
\begin{proof}
	Denote
	\begin{align*}
		\phi_{t}^{k} &:= b_{u}(X_{t}^{k}, 	u_{t}^{k}) Y_{t}^{k}+\sigma_{u}( u_{t}^{k}) Z_{t}^{k}+r_{u}(X_{t}^{k}, u_{t}^{k}),\\
		\phi_{n}^{k} &:= b_{u}(X_{n}^{k,N}, 	u_{t_{n}}^{k}) Y_{n}^{k,N}+\sigma_{u}( u_{t_{n}}^{k}) Z_{n}^{k,N}+r_{u}(X_{n}^{k,N}, u_{t_{n}}^{k}).
	\end{align*}
	For notation simplicity, we shall omit the superscript $k$ in this proof, such as $\phi_{t}^{k}=\phi_{t}$, $\phi_{n}^{k}=\phi_{n}$. Then we have
	\begin{equation}
		\begin{aligned}
			&\int_0^T  (\nabla J_u(u_t) - \nabla J_u^N(u_t))^2 \mathrm{d}t \\
			\leq& 2 \sum_{n=0}^{N-1} \int_{t_n}^{t_{n+1}} \left[ (\nabla J_u(u_t) - \nabla J_u(u_{t_n}))^2 + (\nabla J_u(u_{t_n}) - \nabla J_u^N(t_n, u_{t_n}))^2 \right] \mathrm{d}t \\
			=& 2 \sum_{n=0}^{N-1} \int_{t_n}^{t_{n+1}} (\mathbb{E}[\phi_t - \phi_{t_n}])^2 \mathrm{d}t + 2 \sum_{n=0}^{N-1} \int_{t_n}^{t_{n+1}} (\mathbb{E}[\phi_{t_n} - \phi_n])^2 \mathrm{d}t \\
			=& 2(I_1 + I_2),\\
			\notag 
		\end{aligned}
	\end{equation}
	where
	\begin{equation}
		\begin{aligned}
			I_1 = & \sum_{n=0}^{N-1} \int_{t_n}^{t_{n+1}} (\mathbb{E}[\phi_t - \phi_{t_n}])^2 \mathrm{d}t 
			\leq  \sum_{n=0}^{N-1} \int_{t_n}^{t_{n+1}} \left(  \int_{t_n}^{t} \frac{d}{dr} \mathbb{E}[\phi_r] \Big\Vert_{r=s} \mathrm{d}s \right)^2 \mathrm{d}t\\
			\leq & h \sum_{n=0}^{N-1} \int_{t_n}^{t_{n+1}} \int_{t_n}^{t} \left( \mathbb{E}[\mathcal{L}^0 \bar{\phi} (s,X_s)] \right)^2 \mathrm{d}s \mathrm{d}t \leq \frac{C}{N^2},
		\end{aligned}
	\end{equation}
	with $$\mathcal{L}^0 \bar{\phi}(t, x) = \partial_t  \bar{\phi}(t, x) + b(x, u(t)) \partial_x  \bar{\phi}(t, x) + \frac{1}{2} \sigma(u(t))^2 \partial_{xx}  \bar{\phi}(t, x),$$
		\begin{align*}
			I_2 = & \sum_{n=0}^{N-1} \int_{t_n}^{t_{n+1}} \Big(\mathbb{E}[\varphi_{t_n} - \varphi_n] \Big)^2 \mathrm{d}t \\
			=& \sum_{n=0}^{N-1} \int_{t_n}^{t_{n+1}} \Big(\mathbb{E}[b'_u(X_{t_n}, u_{t_n})Y_{t_n} + \sigma'_u(u_{t_n})Z_{t_n} + r_u(u_{t_n}, X_{t_n})    \\
			& - b'_u(X^N_n, u_{t_n})Y^N_n - \sigma'_u(u_{t_n})Z^N_n - r_u(u_{t_n}, X^N_n)] \Big)^2 \mathrm{d}t \\
			\leq & \sum_{n=0}^{N-1} \int_{t_n}^{t_{n+1}} \Big(\Big(\mathbb{E}[b'_u(X_{t_n}, u_{t_n})Y_{t_n} - b'_u(X^N_n, u_{t_n})Y_{t_n}^{t_n, X_n^N}]\Big)^2 \\
			& + \Big(\mathbb{E}[b'_u(X_{n}^N, u_{t_n})Y_{t_n}^{t_n, X_n^N} - b'_u(X^N_n, u_{t_n})Y^N_n]\Big)^2 \\			
			& + \Big(\mathbb{E}[\sigma'_u(u_{t_n})Z_{t_n} - \sigma'_u(u_{t_n})Z_{t_n}^{t_n, X^N_n}]\Big)^2 + \Big(\mathbb{E}[\sigma'_u(u_{t_n})Z_{t_n}^{t_n, X^N_n} - \sigma'_u(u_{t_n})Z^N_n]\Big)^2\\
			&+ \Big(\mathbb{E}[r_u(u_{t_n}, X_{t_n}) - r_u(u_{t_n}, X^N_n)]\Big)^2\Big) \mathrm{d} t.
			\end{align*}
	As $ |\mathbb{E}[ g(X_{t_n}) -g( X^N_n)]| \leq C h$ for any $g \in C_b^4$, then we have 
	where
	\begin{equation}
		\begin{aligned}
			&\Big(\mathbb{E}[b'_u(X_{t_n}, u_{t_n})Y_{t_n} - b'_u(X^N_n, u_{t_n})Y_{t_n}^{t_n, X_n^N}\Big)^2 \\
			= & \Big(\mathbb{E}[b'_u(X_{t_n}, u_{t_n}) \eta(t_n, X_{t_n}) - b'_u(X^N_n, u_{t_n}) \eta (t_n, X_n^N)\Big)^2
			\leq  \frac{C}{N^2},
			\notag
		\end{aligned}
	\end{equation}
	\begin{equation}
		\begin{aligned}
			&\Big(\mathbb{E}[b'_u(X_{n}^N, u_{t_n})Y_{t_n}^{t_n, X_n^N} - b'_u(X^N_n, u_{t_n})Y^N_n\Big)^2 \\
			\leq & \mathbb{E}[b'_u(X_{n}^N, u_{t_n})^2] \mathbb{E}\Big[(Y_{t_n}^{t_n, X_n^N} - Y^N_n)^2\Big] \leq \frac{C}{N^2},
			\notag
		\end{aligned}
	\end{equation}
	and similarly
	\begin{equation}
		\begin{aligned}
			\Big(\mathbb{E}[\sigma'_u(u_{t_n})Z_{t_n} - \sigma'_u(u_{t_n})Z_{t_n}^{t_n, X^N_n}]\Big)^2&\leq \frac{C}{N^2},\\
			\Big(\mathbb{E}[\sigma'_u(u_{t_n})Z_{t_n}^{t_n, X^N_n} - \sigma'_u(u_{t_n})Z^N_n]\Big)^2&\leq \frac{C}{N^2},\\
			\Big(\mathbb{E}[r_u(u_{t_n}, X_{t_n}) - r_u(u_{t_n}, X^N_n)]\Big)^2\Big)&\leq \frac{C}{N^2}.
			\notag
		\end{aligned}
	\end{equation}
	Then, the desired result follows by combining the estimates above.
\end{proof}
Clearly, $\nabla J_u^N (\cdot)$ depends on the numerical scheme \eqref{1014}. As established in the preceding lemma, this dependence induces an error bound between the numerical approximation $\nabla J_u^N (\cdot)$ and the exact derivative $\nabla J_u (\cdot)$. To analyze the convergence property of the iteration scheme \eqref{1016}, we must quantify the discrepancy between the exact optimal control $u \in \mathcal{K}$ and the optomal control $u^{*, N}$ found in the subspace $\mathcal{K}_N$.
\begin{lemma}\label{lemma4}
	Assume that $u^*$, $\nabla J_u(\cdot)$ is Lipschitz and $\nabla J_u(\cdot)$ uniformly monotone around $u^*$ and $u^{*,N}$ in the sense that there exist positive constants $\lambda$ and $C$ such that
	\begin{equation}
	\begin{aligned}
		\|\nabla J_u(u^*) - \nabla J_u(v)\|_2 &\leqslant C\|u^* - v\|_2, \quad \forall v \in \mathcal{K}, \\
		(\nabla J_u(u^*) - \nabla J_u(v), u^* - v) &\geqslant \lambda\|u^* - v\|_2^2, \quad \forall v \in \mathcal{K}, \\
		\|\nabla J_u(u^{*,N}) - \nabla J_u(v)\|_2 &\leqslant C\|u^{*,N} - v\|_2, \quad \forall v \in {\mathcal{K}}_N, \\
		(\nabla J_u(u^{*,N}) - \nabla J_u(v), u^{*,N} - v) &\geqslant \lambda\|u^{*,N} - v\|_2^2, \quad \forall v \in {\mathcal{K}}_N,
	\end{aligned}
	\end{equation}
 	then the following inequality holds:
 	\begin{equation}
 		\|u^* - u^{*,N}\|_2^2 \leq \frac{C}{N^2},
 	\end{equation}
 	futher
		\begin{equation}
			\|\nabla J_u(u^{*,N}) - \nabla J_u(u^*)\|_2^2 \leq \frac{C}{N^2}.
		\end{equation}
\end{lemma}
\begin{proof}
	\begin{align*}
		\|u^* - u^{*,N}\|_2 &= \|u^* - \mathcal{P}_{{\mathcal{K}}_N}(u^* - \eta \nabla J_u(u^*)) + \mathcal{P}_{\mathcal{K}_N}(u^* - \eta \nabla J_u(u^*)) - u^{*,N}\|_2 \\
		& \leq \|u^* - \mathcal{P}_{{\mathcal{K}}_N}(u^* - \eta \nabla J_u(u^*))\|_2+ \|\mathcal{P}_{{\mathcal{K}}_N}(u^* - \eta \nabla J_u(u^*)) - u^{*,N} \|_2,
	\end{align*}
	where
	\begin{equation}
		\begin{aligned}
			& \|\mathcal{P}_{{\mathcal{K}}_N}(u^* - \eta\nabla J_u(u^*)) - u^{*,N} \|_2^2\\
			\leq & \|\mathcal{P}_{{\mathcal{K}}_N}(u^* - \eta \nabla J_u(u^*)) - \mathcal{P}_{{\mathcal{K}}_N}(u^{*,N} - \eta \nabla J_u(u^{*,N}))\|_2^2\\
			\leq & \|\mathcal{P}_{{\mathcal{K}}_N}(u^* - u^{*,N}- \eta \nabla J_u(u^*) + \eta \nabla J_u(u^{*,N}))\| _2^2\\
			\leq & \| u^* - u^{*,N} \|_2^2 - 2 \eta \langle u^{*} - u^{*,N},  \nabla J_{u} (u^{*}) - \nabla J_{u}(u^{*,N}) \rangle \\
			& + \eta^2 \|  \nabla J_{u} (u^{*}) - \nabla J_{u}(u^{*,N}) \|_2^2 \\
			\leq & \| u^* - u^{*,N} \|_2^2 -2\lambda \eta  \| u^* - u^{*,N} \|_2^2 + C \eta^2 \| u^* - u^{*,N} \|_2^2,
			\notag
		\end{aligned}
	\end{equation}
	so
	\begin{align*}
		\|u^* - u^{*,N}\|_2 \leq \|u^* - \mathcal{P}_{{\mathcal{K}}_N}(u^* - \eta \nabla J_u(u^*))\|_2 + \sqrt{1 - 2\lambda\eta + C\eta^2} \|u^* - u^{*,N}\|_2.
	\end{align*}
	Let $\eta = \lambda/C$, $C_2 = \left(1 - \sqrt{1 - 2\lambda\eta + C\eta^2}\right)^{-1}$, we have
	\[
	\|u^* - u^{*,N}\|_2 \leqslant C_2 \|u^* - \mathcal{P}_{\mathcal{K}_N}(u^* - \eta \nabla J_u(u^*))\|_2.
	\]
	Since $\mathcal{C}$ is invariant in time, for $v \in \mathcal{U}_N$, it holds that $\mathcal{P}_\mathcal{K} v \in \mathcal{U}_N$. Thus we have $\mathcal{P}_\mathcal{K} v \in \mathcal{K}_N$, and then we have $\mathcal{P}_\mathcal{K} v = \mathcal{P}_{\mathcal{K}_N} v$. Now, denoting $\omega := u^* - \eta \nabla J_u (u^*)$, we have
	\begin{align*}
		\|u^* - u^{*,N}\|_2 &\leqslant C_2 \|u^* - \mathcal{P}_{{\mathcal{K}}_N}(u^* - \eta \nabla J_u (u^*))\|_2 = C_2 \|\mathcal{P}_\mathcal{K} \omega - \mathcal{P}_{\mathcal{K}_N} \omega\|_2 \\
		&\leqslant C_2 \left( \|\mathcal{P}_\mathcal{K} \omega - \mathcal{P}_\mathcal{K} P_{\mathcal{U}_N} \omega\|_2 + \|\mathcal{P}_\mathcal{K} \mathcal{P}_{\mathcal{U}_N} \omega - \mathcal{P}_{\mathcal{K}_N} \omega\|_2 \right) \\
		&= C_2 \left( \|\mathcal{P}_\mathcal{K} \omega - \mathcal{P}_\mathcal{K} \mathcal{P}_{\mathcal{U}_N} \omega\|_2 + \|\mathcal{P}_{\mathcal{K}_N} \mathcal{P}_{\mathcal{U}_N} \omega - \mathcal{P}_{\mathcal{K}_N} \omega\|_2 \right) \\
		&\leqslant 2C_2 \|\omega - \mathcal{P}_{\mathcal{U}_N} \omega\|_2.
	\end{align*}
	As $u^*$ is Lipschitz, $\nabla J_u(\cdot)$ is Lipschitz around $u^*$, we have $\|\omega - \mathcal{P}_{\mathcal{U}_N}\omega\|_2 \leq \frac{C}{N}$, and thus $\|u^* - u^{*,N}\|_2 \leq \frac{C}{N}$. Then, the conclusion follows:
	\begin{equation}\label{c2}
		\|\nabla J_u(u^{*,N}) - \nabla J_u(u^*)\|_2^2 \leq C\|u^* - u^{*,N}\|_2^2 \leq \frac{C}{N^2}.
	\end{equation}
\end{proof}
With the conclusion of Lemma \ref{lemma4} in hand, we also need to deduce the error between $u^{K+1}$ and the optimal control in the piece-wise constant subset $\mathcal{K}_N$.
\begin{lemma}\label{thm3}
	Assume all the assumptions in Lemma \ref{lemma3} and Lemma \ref{lemma4} are true. Let $\eta_{k} = \frac{\theta}{k+M}$ for some constants $\theta$ and $M$ such that $\lambda\theta - 4C_{L}\theta^{2}/(1 + M)>1$. Also, let $\{u^{k}\}_k$ be the sequence of estimated optimal control obtained by the SGD optimization scheme \eqref{1016}. Then the following inequality holds:
	\begin{equation}\label{c1}
		\mathbb{E} \left[ {\big\Vert u^{K+1} - u^{*,N} \big\Vert}_2^2 \right] 
		\leq C \left( \frac{N}{K} + \frac{1}{N^2} \right).
		\notag 
	\end{equation}
\end{lemma}
\begin{proof}
	By equation \eqref{17}, the following holds for any positive $\eta_{k}$
	\begin{equation}\label{1033}
		u^{*,N} = \mathcal{P}_{\mathcal{K}_N} \left(u^{*,N}-\eta_{k}\nabla J_{u}(u^{*,N})\right).
	\end{equation}
	Subtracting \eqref{1033} from both sides of \eqref{1016}, we have
	\begin{equation}
		{\big\Vert u^{k+1} - u^{*,N} \big\Vert}_2^2 = {\big\Vert \mathcal{P}_{\mathcal{K}_N} \left( u^{k} - u^{*,N} -  \left( \eta_{k}\nabla j_{u}^k (u^{k}) - \eta_{k}\nabla J_{u}(u^{*,N}) \right) \right)   \big\Vert}_2^2.
	\end{equation}	
	Taking conditional expectation $\mathbb{E}^{k}[\cdot]$ to the above equation, we obtain
	\begin{equation}\label{35}
		\begin{aligned}
			\mathbb{E}^{k} \left[ {\big\Vert u^{k+1} - u^{*,N} \big\Vert}_2^2 \right]  \leq & {\big\Vert u^{k} - u^{*,N} \big\Vert}_2^2 - 
			2 \eta_k \langle u^{k} - u^{*,N}, \mathbb{E}^{k} \left[ \nabla j_{u}^k (u^{k}) - \nabla J_{u}(u^{*,N}) \right]\rangle  \\
			& + {\eta_k}^2 \mathbb{E}^{k} 
			\left[ {\big\Vert \nabla j_{u}^k (u^{k}) - \nabla J_{u}(u^{*,N}) \big\Vert}_2^2 \right].
		\end{aligned}
	\end{equation}
	From the convexity assumption and Lemma \ref{lemma3}, we deduce, from Young's inequality with $\lambda/2$, and the fact that $u^{k}$ is $\mathcal{G}_{k}$ measurable, that
	\begin{equation}\label{36}
		\begin{aligned}
			& 
			- \langle u^{k} - u^{*,N}, \mathbb{E}^{k} \left[ \nabla j_{u}^k (u^{k}) - \nabla J_{u}(u^{*,N}) \right]\rangle = - \langle u^{k} - u^{*,N}, \nabla J_{u}^N (u^{k}) - \nabla J_{u}(u^{*,N}) \rangle \\
			\leq & - \langle u^{k} - u^{*,N}, \nabla J_{u}^N (u^{k}) - \nabla J_{u}(u^k) \rangle - \langle u^{k} - u^{*,N}, \nabla J_{u}(u^{k}) - \nabla J_{u}(u^{*,N}) \rangle\\
			\leq & \frac{1}{2\lambda} {\big\Vert \nabla J_{u}^N (u^{k}) - \nabla J_{u}(u^{k}) \big\Vert}^2 + \frac{\lambda}{2} {\big\Vert u^{k} - u^{*,N} \big\Vert}_2^2 - \lambda {\big\Vert u^{k} - u^{*,N} \big\Vert}_2^2\\
			\leq & \frac{1}{2\lambda} \frac{C}{N^2} - \frac{\lambda}{2} {\big\Vert u^{k} - u^{*,N} \big\Vert}_2^2.
		\end{aligned}
	\end{equation}
	Moreover, from Lemma \ref{lemma2}, Lemma \ref{lemma3}, and the convexity assumption, we have
	\begin{equation}\label{37}
		\begin{aligned}
			&\mathbb{E}^{k} 
			\left[ {\big\Vert \nabla j_{u}^k (u^{k})  - \nabla J_{u}^N (u^{k}) + \nabla J_{u}^N (u^{k}) - \nabla J_{u}(u^{*,N}) \big\Vert}_2^2 \right]\\
			\leq & 2 \mathbb{E}^{k} 
			\left[ {\big\Vert \nabla J_{u}^N (u^{k}) - \nabla J_{u}(u^{*,N}) \big\Vert}_2^2 \right] + CN\\
			\leq & 4 \left( \mathbb{E}^{k} 
			\left[ {\big\Vert \nabla J_{u}^N (u^{k}) - \nabla J_{u}(u^{k}) \big\Vert}_2^2 \right] + \mathbb{E}^{k} 
			\left[ {\big\Vert \nabla J_{u} (u^{k}) - \nabla J_{u}(u^{*,N}) \big\Vert}_2^2 \right] \right) + CN\\
			\leq & 4 \left( \frac{C}{N^2} + C_L {\big\Vert u^{k} - u^{*,N} \big\Vert}_2^2 \right) + CN.
		\end{aligned}
	\end{equation}
	Inserting \eqref{36}--\eqref{37} in \eqref{35}, we obtain
	\begin{equation}
		\begin{aligned}
			\mathbb{E}^{k} \left[ {\big\Vert u^{k+1} - u^{*,N} \big\Vert}_2^2 \right]  \leq & {\big\Vert u^{k} - u^{*,N} \big\Vert}_2^2 - 2 \eta_k \langle u^{k} - u^{*,N}, \mathbb{E}^{k} \left[ \nabla j_{u}^k (u^{k}) - \nabla J_{u}(u^{*,N}) \right]\rangle \\
			& + {\eta_k}^2 \mathbb{E}^{k} 
			\left[ {\big\Vert \nabla j_{u}^k (u^{k}) - \nabla J_{u}(u^{*,N}) \big\Vert}_2^2 \right]\\
			\leq & {\big\Vert u^{k} - u^{*,N} \big\Vert}_2^2 + \frac{\eta_{k}}{\lambda} \frac{C}{N^2} - \lambda \eta_{k} {\big\Vert u^{k} - u^{*,N} \big\Vert}_2^2 \\
			& + 4 \eta_{k}^2 \left( \frac{C}{N^2} + C_L {\big\Vert u^{k} - u^{*, N} \big\Vert}_2^2 \right) + \eta_{k}^2 CN\\
			= & (1 - c_k \eta_{k}) {\big\Vert u^{k} - u^{*,N} \big\Vert}_2^2 + \eta_{k}^2 CN + \left( \frac{\eta_k}{
			\lambda} + 4 \eta_{k}^2 \right) \frac{C}{N^2},
		\end{aligned}
	\end{equation}
	where $c_{k}:= \lambda - 4C_{L}\eta_{k}$.
	Let $\tilde{\eta}_k=\frac{1}{k + M}$. We can find $\theta$ and $M$ such that
	\[
	c_l:=\lambda\theta - 4C_{L}\frac{\theta^{2}}{1 + M}>1,
	\]
	and we have that, when $k$ is large enough, $c_l\tilde{\eta}_{k} \leq c_{k}\eta_{k}$ for $\eta_{k}=\frac{\theta}{k+M}$. 
	\begin{equation}
		\begin{aligned}
			\mathbb{E}^{k} \left[ {\big\Vert u^{k+1} - u^{*,N} \big\Vert}_2^2 \right] \leq & (1 - c_l \tilde{\eta}_k) {\big\Vert u^{k} - u^{*,N} \big\Vert}_2^2 + C \left( \tilde{\eta}_k^2 N + \frac{\tilde{\eta}_k}{N^2} \right).
		\end{aligned}
	\end{equation}
	Next, we take expectation $\mathbb{E}[\cdot]$ to both sides of the above estimate and apply it recursively from $k = 0$ to $k = K$ to get
		\begin{align*}
			\mathbb{E} \left[ {\big\Vert u^{K+1} - u^{*} \big\Vert}_2^2 \right] &\leq\prod_{k = 0}^{K} (1-c_l \tilde{\eta}_k)  \mathbb{E}\left[\big\Vert u^{0}-u^* \big\Vert_2^{2}\right] + \left(\sum_{m = 1}^{K} \tilde{\eta}_{m-1} \prod_{k = m}^{K}(1-c_l \tilde{\eta}_k)\right)\frac{C}{N^2}\\
			& \quad +\frac{C\tilde{\eta}_{K}}{N^2} + \left(\sum_{m = 1}^{K} \tilde{\eta}_{m-1}^2 \prod_{k = m}^{K}(1-c_l \tilde{\eta}_k)\right)CN + C \tilde{\eta}_{K}^2 N\\
			&\leq(K + M)^{-c_l} {\big\Vert u^{0} - u^{*} \big\Vert}_2^2\\
			& \quad +C N \left((K + M)^{-1} - \frac{(1 + M)^{c_l-1}}{(K+M)^{c_l}}\right)+\frac{C}{N^2}.
		\end{align*}
	Since $c_l>1$ and $\prod_{k = m}^{K}(1 - c_l\tilde{\eta}_k)\sim O((K/m)^{-c_l})$, the above estimate gives us
	\begin{equation}
	\mathbb{E} \left[ {\big\Vert u^{K+1} - u^{*,N} \big\Vert}_2^2 \right] \leq C\left(\frac{N}{K}+\frac{1}{N^2}\right).
	\end{equation}	
\end{proof}
Now we are ready to prove the main convergence result of the iteration scheme \eqref{1016} under the convexity assumption, i.e., the convergence between $u^{K+1}$ and the exact optimal control $u^* \in \mathcal{K}$.
\begin{theorem}
	Assume that all the assumptions hold in Lemma \ref{thm3}, and assume the optimal control $u^*$ is bounded. Then we have the following convergence result:
	\begin{equation}
		\mathbb{E}\left[\|u^{K+1} - u^*\|_2^2\right] \leq C \left(\frac{N}{K} + \frac{1}{N^2}\right).
	\end{equation}
\end{theorem}
\begin{proof}
	From the Lemma \ref{thm3} and the fact that 
	\eqref{c2}, we have
	\begin{align*}
		\mathbb{E}\left[\|u^{K+1} - u^*\|_2^2\right] &\leq 2 \mathbb{E}\left[\|u^{K+1} - u^{*,N}\|_2^2\right] + 2 \mathbb{E}\left[\|u^{*,N} - u^*\|_2^2\right] \\
		&\leq C \left(\frac{N}{K} + \frac{1}{N^2}\right) + \frac{C}{N^2}
	\end{align*}
	as desired.
\end{proof}


%
\begin{remark}
	The error estimate reveals the interplay among three key factors:(i) the 	iteration count $K$ in SGD, (ii) the depth of the corresponding Neural SDE, 	and (iii) the discretization error of approximating the FBSDE. Specifically, by choosing
	$K = cN^{3}$, where $c$ is a constant, the numerical scheme \eqref{1014} achieves first-order convergence $(O(\frac{1}{N}))$.
\end{remark}

\section{Numerical examples}
In this section, we consider several numerical examples to illustrate the performance of our high-order backpropagation algorithm for Neural SDE. 
\begin{example}\label{ex1}
	Our first example is from \cite{archibald2024numerical}. The optimal control problem is stated as 
	\begin{equation}
		J(u^*) = \min_{u \in \mathcal{K}} J(u),
		\notag
	\end{equation}
	with the cost function
	\begin{equation}
		J(u) = \frac{1}{2} \int_0^1 \mathbb{E} \left[ |X_t - X_t^*|^2 \right] \mathrm{d}t + \frac{1}{2} \int_0^1 |u_t|^2 \mathrm{d}t + \frac{1}{2} |X_T|^2,
		\notag
	\end{equation}
	and the controlled state process 
	\begin{equation}
		dX_t = (u_t - a_t)\mathrm{d}t + \sigma u_t \mathrm{d}W_t,
		\notag
	\end{equation}
	where the vector function 
	$
	a_t = \left[ \dfrac{-t^2}{2\beta_t}, \dfrac{-\sin t}{\beta_t} \right]^\top,
	$
	$\beta_t = (1 + \sigma^2) + \sigma^2(1-t)$, and $\sigma$ is a constant.
	The deterministic function $X_t^*$ is given by
	\[
	X_t^* := \left[
		t + \alpha_t \frac{0.5 - X_T^{1}}{\sigma^2}, \cos t + \alpha_t \frac{\sin 1 - X_T^{2}}{\sigma^2}
	\right]^\top,
	\]
	where $\alpha_t = \ln \dfrac{1 + 2\sigma^2}{\sigma^2(2-t) + 1}$. For $D := \dfrac{\ln(1 + \frac{\sigma^2}{1 + \sigma^2})}{\sigma^2 + \ln(1 + \frac{\sigma^2}{1 + \sigma^2})}$, $X_T = \left[ X_T^{1}, X_T^{2}\right]^\top$ is defined as $X_T := [D/2, D \cdot \sin 1]^\top$. And the corresponding exact optimal control is
	\[
	u_t^* := \left[
		\frac{-t^2/2 + T^2/2 - X_T^{1}}{\beta_t}, \frac{-\sin t + \sin 1 - X_T^{2}}{\beta_t}
	\right]^\top.
	\]
	\begin{figure}[htbp] 
		\centering
		\includegraphics[width=1\textwidth]{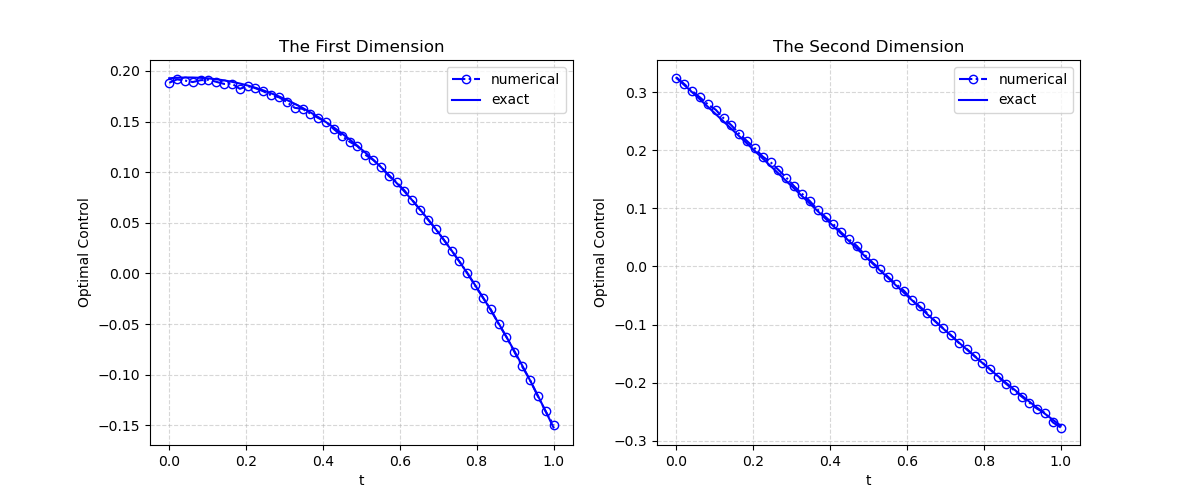}
		\caption{Exact solution and numerical solution.}\label{figure1}
	\end{figure}
	\begin{figure}[htbp] 
		\centering
		\includegraphics[width=0.6\textwidth]{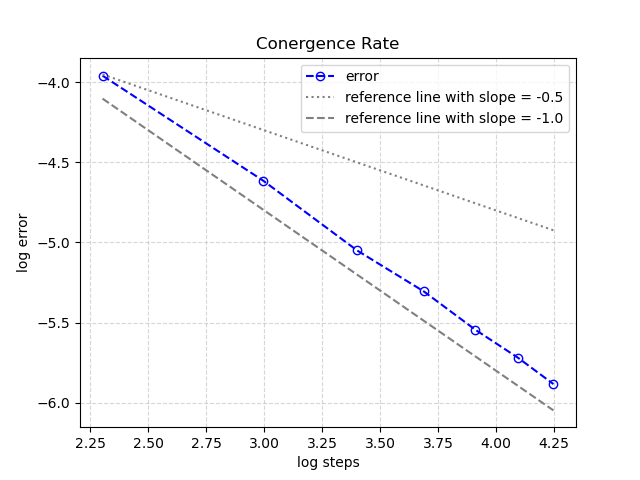}
		\caption{Convergence with respect to $N$.}\label{figure11}
	\end{figure}

	We set $X_0 = 0$, $T = 1$, $\sigma = 0.5$, iteration steps $K = 0.2 \times N^3$ for each $N$. The Figure \ref{figure1} shows that the numerical solutions matches the exact solutions very well when N = 50. In the Figure \ref{figure11}, the depth of neural networks is chosen as $N = 20, 30, 40, \ldots, 70$, 
	we solve the above SOCP 30 times, and it gives the root mean square errors (RMSEs) plotted against $N$ (presented by $\log N$ on the $x$-axis). As can be seen from Figure \ref{figure1} and \ref{figure11}, the convergence order of our high-order algorithm can reach 1.
\end{example}
\begin{example}
	\begin{figure}[htbp] 
		\centering
		\includegraphics[width=1\textwidth]{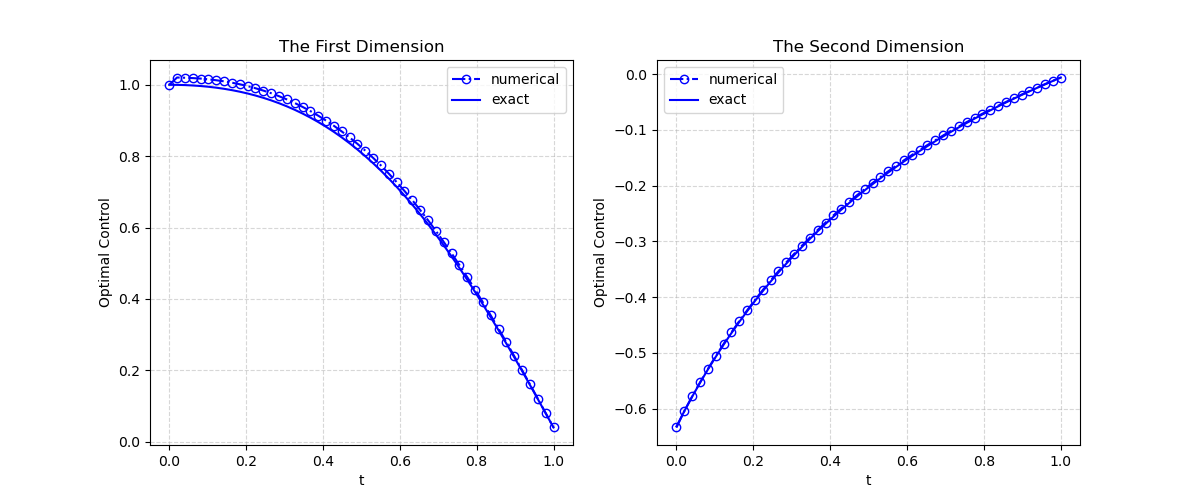}
		\caption{Exact solution and numerical solution.}\label{figure21}
	\end{figure}
	\begin{figure}[htbp] 
		\centering
		\includegraphics[width=0.6\textwidth]{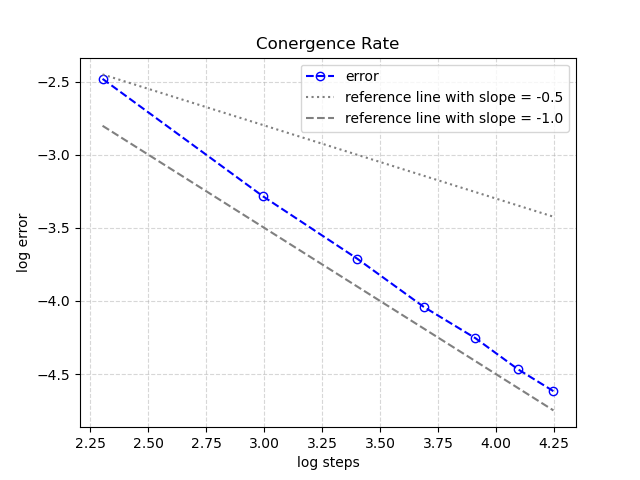}
		\caption{Convergence with respect to $N$.}\label{figure2}
	\end{figure}
	The second example has been used in \cite{IJNAM-10-4}, which is the Black-Scholes type of optimal control problems:
	\[ \min_{u \in \mathcal{K}} J(u) = \frac{1}{2} \int_0^T \mathbb{E} \left[ (X_t - X^*_t)^2 \right] dt + \frac{1}{2} \int_0^T |u_t|^2 dt, \]
	with the controlled state equation
	\[ dX_t = u(t)X_t \, dt + \sigma X_t \, dW_t. \]
	Here $\sigma$ is a constant. The deterministic function $X_t^*$ and the corresponding exact solution $u_t^*$ are given by
	\[
	X_t^* := \left[
	\frac{e^{\sigma^2 t} - (T - t)^2}{\frac{1}{x_0} - Tt + \frac{t^2}{2}} + 1,\quad \frac{e^{\sigma^2 t} - (e^{-T} - e^{-t})^2}{\frac{1}{x_0} + 1 - e^{-t} - t e^{-T}} - e^{-t}
	\right]^\top,
	\]
	\[
	u_t^* := \left[
	\frac{T - t}{\frac{1}{x_0} - Tt + \frac{t^2}{2}},\quad \frac{e^{-T} - e^{-t}}{\frac{1}{x_0} + 1 - e^{-t} - t e^{-T}}
	\right]^\top.
	\]
	We set $x_0 = 1$, $T = 1$ and $\sigma = 0.1$. The same training settings for the Neural SDE are used. Numerical results by our high-order backpropagation algorithm are presented in Figure \ref{figure21} and Figure \ref{figure2}. Similar conclusions can be made as for Example \ref{ex1}. The method converges with the first order accuracy.
\end{example}
\bibliography{sn-bibliography}


\begin{thebibliography}{29}
\ifx \bisbn   \undefined \def \bisbn  #1{ISBN #1}\fi
\ifx \binits  \undefined \def \binits#1{#1}\fi
\ifx \bauthor  \undefined \def \bauthor#1{#1}\fi
\ifx \batitle  \undefined \def \batitle#1{#1}\fi
\ifx \bjtitle  \undefined \def \bjtitle#1{#1}\fi
\ifx \bvolume  \undefined \def \bvolume#1{\textbf{#1}}\fi
\ifx \byear  \undefined \def \byear#1{#1}\fi
\ifx \bissue  \undefined \def \bissue#1{#1}\fi
\ifx \bfpage  \undefined \def \bfpage#1{#1}\fi
\ifx \blpage  \undefined \def \blpage #1{#1}\fi
\ifx \burl  \undefined \def \burl#1{\textsf{#1}}\fi
\ifx \doiurl  \undefined \def \doiurl#1{\url{https://doi.org/#1}}\fi
\ifx \betal  \undefined \def \betal{\textit{et al.}}\fi
\ifx \binstitute  \undefined \def \binstitute#1{#1}\fi
\ifx \binstitutionaled  \undefined \def \binstitutionaled#1{#1}\fi
\ifx \bctitle  \undefined \def \bctitle#1{#1}\fi
\ifx \beditor  \undefined \def \beditor#1{#1}\fi
\ifx \bpublisher  \undefined \def \bpublisher#1{#1}\fi
\ifx \bbtitle  \undefined \def \bbtitle#1{#1}\fi
\ifx \bedition  \undefined \def \bedition#1{#1}\fi
\ifx \bseriesno  \undefined \def \bseriesno#1{#1}\fi
\ifx \blocation  \undefined \def \blocation#1{#1}\fi
\ifx \bsertitle  \undefined \def \bsertitle#1{#1}\fi
\ifx \bsnm \undefined \def \bsnm#1{#1}\fi
\ifx \bsuffix \undefined \def \bsuffix#1{#1}\fi
\ifx \bparticle \undefined \def \bparticle#1{#1}\fi
\ifx \barticle \undefined \def \barticle#1{#1}\fi
\bibcommenthead
\ifx \bconfdate \undefined \def \bconfdate #1{#1}\fi
\ifx \botherref \undefined \def \botherref #1{#1}\fi
\ifx \url \undefined \def \url#1{\textsf{#1}}\fi
\ifx \bchapter \undefined \def \bchapter#1{#1}\fi
\ifx \bbook \undefined \def \bbook#1{#1}\fi
\ifx \bcomment \undefined \def \bcomment#1{#1}\fi
\ifx \oauthor \undefined \def \oauthor#1{#1}\fi
\ifx \citeauthoryear \undefined \def \citeauthoryear#1{#1}\fi
\ifx \endbibitem  \undefined \def \endbibitem {}\fi
\ifx \bconflocation  \undefined \def \bconflocation#1{#1}\fi
\ifx \arxivurl  \undefined \def \arxivurl#1{\textsf{#1}}\fi
\csname PreBibitemsHook\endcsname

\bibitem[\protect\citeauthoryear{Jia and Benson}{2019}]{jia2019neural}
\begin{botherref}
\oauthor{\bsnm{Jia}, \binits{J.}},
\oauthor{\bsnm{Benson}, \binits{A.R.}}:
Neural jump stochastic differential equations.
Advances in Neural Information Processing Systems
\textbf{32}
(2019)
\end{botherref}
\endbibitem

\bibitem[\protect\citeauthoryear{Kidger et~al.}{2021}]{kidger2021efficient}
\begin{barticle}
\bauthor{\bsnm{Kidger}, \binits{P.}},
\bauthor{\bsnm{Foster}, \binits{J.}},
\bauthor{\bsnm{Li}, \binits{X.C.}},
\bauthor{\bsnm{Lyons}, \binits{T.}}:
\batitle{Efficient and accurate gradients for neural sdes}.
\bjtitle{Advances in Neural Information Processing Systems}
\bvolume{34},
\bfpage{18747}--\blpage{18761}
(\byear{2021})
\end{barticle}
\endbibitem

\bibitem[\protect\citeauthoryear{Kong et~al.}{2020}]{kong2020sde}
\begin{botherref}
\oauthor{\bsnm{Kong}, \binits{L.}},
\oauthor{\bsnm{Sun}, \binits{J.}},
\oauthor{\bsnm{Zhang}, \binits{C.}}:
Sde-net: Equipping deep neural networks with uncertainty estimates.
arXiv preprint arXiv:2008.10546
(2020)
\end{botherref}
\endbibitem

\bibitem[\protect\citeauthoryear{Liu et~al.}{2020}]{liu2020does}
\begin{bchapter}
\bauthor{\bsnm{Liu}, \binits{X.}},
\bauthor{\bsnm{Xiao}, \binits{T.}},
\bauthor{\bsnm{Si}, \binits{S.}},
\bauthor{\bsnm{Cao}, \binits{Q.}},
\bauthor{\bsnm{Kumar}, \binits{S.}},
\bauthor{\bsnm{Hsieh}, \binits{C.-J.}}:
\bctitle{How does noise help robustness? explanation and exploration under the
  neural sde framework}.
In: \bbtitle{Proceedings of the IEEE/CVF Conference on Computer Vision and
  Pattern Recognition},
pp. \bfpage{282}--\blpage{290}
(\byear{2020})
\end{bchapter}
\endbibitem

\bibitem[\protect\citeauthoryear{Guo et~al.}{2017}]{guo2017calibration}
\begin{bchapter}
\bauthor{\bsnm{Guo}, \binits{C.}},
\bauthor{\bsnm{Pleiss}, \binits{G.}},
\bauthor{\bsnm{Sun}, \binits{Y.}},
\bauthor{\bsnm{Weinberger}, \binits{K.Q.}}:
\bctitle{On calibration of modern neural networks}.
In: \bbtitle{International Conference on Machine Learning},
pp. \bfpage{1321}--\blpage{1330}
(\byear{2017}).
\bcomment{PMLR}
\end{bchapter}
\endbibitem

\bibitem[\protect\citeauthoryear{Kwon et~al.}{2020}]{kwon2020uncertainty}
\begin{barticle}
\bauthor{\bsnm{Kwon}, \binits{Y.}},
\bauthor{\bsnm{Won}, \binits{J.-H.}},
\bauthor{\bsnm{Kim}, \binits{B.J.}},
\bauthor{\bsnm{Paik}, \binits{M.C.}}:
\batitle{Uncertainty quantification using bayesian neural networks in
  classification: Application to biomedical image segmentation}.
\bjtitle{Computational Statistics \& Data Analysis}
\bvolume{142},
\bfpage{106816}
(\byear{2020})
\end{barticle}
\endbibitem

\bibitem[\protect\citeauthoryear{Eaton-Rosen et~al.}{2018}]{eaton2018towards}
\begin{bchapter}
\bauthor{\bsnm{Eaton-Rosen}, \binits{Z.}},
\bauthor{\bsnm{Bragman}, \binits{F.}},
\bauthor{\bsnm{Bisdas}, \binits{S.}},
\bauthor{\bsnm{Ourselin}, \binits{S.}},
\bauthor{\bsnm{Cardoso}, \binits{M.J.}}:
\bctitle{Towards safe deep learning: accurately quantifying biomarker
  uncertainty in neural network predictions}.
In: \bbtitle{Medical Image Computing and Computer Assisted Intervention--MICCAI
  2018: 21st International Conference, Granada, Spain, September 16-20, 2018,
  Proceedings, Part I},
pp. \bfpage{691}--\blpage{699}
(\byear{2018}).
\bcomment{Springer}
\end{bchapter}
\endbibitem

\bibitem[\protect\citeauthoryear{Lakshminarayanan
  et~al.}{2017}]{lakshminarayanan2017simple}
\begin{botherref}
\oauthor{\bsnm{Lakshminarayanan}, \binits{B.}},
\oauthor{\bsnm{Pritzel}, \binits{A.}},
\oauthor{\bsnm{Blundell}, \binits{C.}}:
Simple and scalable predictive uncertainty estimation using deep ensembles.
Advances in neural information processing systems
\textbf{30}
(2017)
\end{botherref}
\endbibitem

\bibitem[\protect\citeauthoryear{Gal and Ghahramani}{2016}]{gal2016dropout}
\begin{bchapter}
\bauthor{\bsnm{Gal}, \binits{Y.}},
\bauthor{\bsnm{Ghahramani}, \binits{Z.}}:
\bctitle{Dropout as a bayesian approximation: Representing model uncertainty in
  deep learning}.
In: \bbtitle{International Conference on Machine Learning},
pp. \bfpage{1050}--\blpage{1059}
(\byear{2016}).
\bcomment{PMLR}
\end{bchapter}
\endbibitem

\bibitem[\protect\citeauthoryear{Chen et~al.}{2018}]{chen2018neural}
\begin{botherref}
\oauthor{\bsnm{Chen}, \binits{R.T.}},
\oauthor{\bsnm{Rubanova}, \binits{Y.}},
\oauthor{\bsnm{Bettencourt}, \binits{J.}},
\oauthor{\bsnm{Duvenaud}, \binits{D.K.}}:
Neural ordinary differential equations.
Advances in neural information processing systems
\textbf{31}
(2018)
\end{botherref}
\endbibitem

\bibitem[\protect\citeauthoryear{Li et~al.}{2020}]{li2020scalable}
\begin{bchapter}
\bauthor{\bsnm{Li}, \binits{X.}},
\bauthor{\bsnm{Wong}, \binits{T.-K.L.}},
\bauthor{\bsnm{Chen}, \binits{R.T.}},
\bauthor{\bsnm{Duvenaud}, \binits{D.}}:
\bctitle{Scalable gradients for stochastic differential equations}.
In: \bbtitle{International Conference on Artificial Intelligence and
  Statistics},
pp. \bfpage{3870}--\blpage{3882}
(\byear{2020}).
\bcomment{Proceedings of Machine Learning Research}
\end{bchapter}
\endbibitem

\bibitem[\protect\citeauthoryear{Archibald}{2020}]{archibald2020stochastic}
\begin{botherref}
\oauthor{\bsnm{Archibald}, \binits{R.}}:
A stochastic gradient descent approach for stochastic optimal control.
East Asian Journal on Applied Mathematics
\textbf{10}(4)
(2020)
\end{botherref}
\endbibitem

\bibitem[\protect\citeauthoryear{Archibald et~al.}{2022}]{MR4470545}
\begin{barticle}
\bauthor{\bsnm{Archibald}, \binits{R.}},
\bauthor{\bsnm{Bao}, \binits{F.}},
\bauthor{\bsnm{Cao}, \binits{Y.}},
\bauthor{\bsnm{Zhang}, \binits{H.}}:
\batitle{A backward {SDE} method for uncertainty quantification in deep
  learning}.
\bjtitle{Discrete and Continuous Dynamical Systems. Series S}
\bvolume{15}(\bissue{10}),
\bfpage{2807}--\blpage{2835}
(\byear{2022})
\end{barticle}
\endbibitem

\bibitem[\protect\citeauthoryear{Archibald
  et~al.}{2024}]{archibald2024numerical}
\begin{barticle}
\bauthor{\bsnm{Archibald}, \binits{R.}},
\bauthor{\bsnm{Bao}, \binits{F.}},
\bauthor{\bsnm{Cao}, \binits{Y.}},
\bauthor{\bsnm{Sun}, \binits{H.}}:
\batitle{Numerical analysis for convergence of a sample-wise backpropagation
  method for training stochastic neural networks}.
\bjtitle{SIAM Journal on Numerical Analysis}
\bvolume{62}(\bissue{2}),
\bfpage{593}--\blpage{621}
(\byear{2024})
\end{barticle}
\endbibitem

\bibitem[\protect\citeauthoryear{Ma and Yong}{1999}]{ma1999forward}
\begin{bchapter}
\bauthor{\bsnm{Ma}, \binits{J.}},
\bauthor{\bsnm{Yong}, \binits{J.}}:
\bctitle{Forward-backward stochastic differential equations and their
  applications-introduction}.
In: \bbtitle{Forward-backward Stochastic Differential Equations and Their
  Applications},
pp. \bfpage{1}--\blpage{24}.
\bpublisher{Springer},
\blocation{Berlin Heidelberg}
(\byear{1999})
\end{bchapter}
\endbibitem

\bibitem[\protect\citeauthoryear{Zhang}{2017}]{zhang2017backward}
\begin{bchapter}
\bauthor{\bsnm{Zhang}, \binits{J.}}:
\bctitle{Backward stochastic differential equations}.
In: \bbtitle{Backward Stochastic Differential Equations: From Linear to Fully
  Nonlinear Theory},
pp. \bfpage{79}--\blpage{99}.
\bpublisher{Springer},
\blocation{New York}
(\byear{2017})
\end{bchapter}
\endbibitem

\bibitem[\protect\citeauthoryear{Cvitanic and
  Zhang}{2005}]{cvitanic2005steepest}
\begin{botherref}
\oauthor{\bsnm{Cvitanic}, \binits{J.}},
\oauthor{\bsnm{Zhang}, \binits{J.}}:
The steepest descent method for forward-backward sdes
(2005)
\end{botherref}
\endbibitem

\bibitem[\protect\citeauthoryear{Delarue and
  Menozzi}{2006}]{delarue2006forward}
\begin{botherref}
\oauthor{\bsnm{Delarue}, \binits{F.}},
\oauthor{\bsnm{Menozzi}, \binits{S.}}:
A forward--backward stochastic algorithm for quasi-linear pdes
(2006)
\end{botherref}
\endbibitem

\bibitem[\protect\citeauthoryear{Douglas~Jr
  et~al.}{1996}]{douglas1996numerical}
\begin{barticle}
\bauthor{\bsnm{Douglas~Jr}, \binits{J.}},
\bauthor{\bsnm{Ma}, \binits{J.}},
\bauthor{\bsnm{Protter}, \binits{P.}}:
\batitle{Numerical methods for forward-backward stochastic differential
  equations}.
\bjtitle{The Annals of Applied Probability}
\bvolume{6}(\bissue{3}),
\bfpage{940}--\blpage{968}
(\byear{1996})
\end{barticle}
\endbibitem

\bibitem[\protect\citeauthoryear{Ma et~al.}{2008}]{ma2008numerical}
\begin{barticle}
\bauthor{\bsnm{Ma}, \binits{J.}},
\bauthor{\bsnm{Shen}, \binits{J.}},
\bauthor{\bsnm{Zhao}, \binits{Y.}}:
\batitle{On numerical approximations of forward-backward stochastic
  differential equations}.
\bjtitle{SIAM Journal on Numerical Analysis}
\bvolume{46}(\bissue{5}),
\bfpage{2636}--\blpage{2661}
(\byear{2008})
\end{barticle}
\endbibitem

\bibitem[\protect\citeauthoryear{Zhao et~al.}{2014a}]{zhao2014new}
\begin{barticle}
\bauthor{\bsnm{Zhao}, \binits{W.}},
\bauthor{\bsnm{Fu}, \binits{Y.}},
\bauthor{\bsnm{Zhou}, \binits{T.}}:
\batitle{New kinds of high-order multistep schemes for coupled forward backward
  stochastic differential equations}.
\bjtitle{SIAM Journal on Scientific Computing}
\bvolume{36}(\bissue{4}),
\bfpage{1731}--\blpage{1751}
(\byear{2014})
\end{barticle}
\endbibitem

\bibitem[\protect\citeauthoryear{Zhao et~al.}{2014b}]{zhao2014second}
\begin{barticle}
\bauthor{\bsnm{Zhao}, \binits{W.}},
\bauthor{\bsnm{Li}, \binits{Y.}},
\bauthor{\bsnm{Fu}, \binits{Y.}}:
\batitle{Second-order schemes for solving decoupled forward backward stochastic
  differential equations}.
\bjtitle{Science China Mathematics}
\bvolume{57},
\bfpage{665}--\blpage{686}
(\byear{2014})
\end{barticle}
\endbibitem

\bibitem[\protect\citeauthoryear{Zhao et~al.}{2014c}]{zhao2014numerical}
\begin{barticle}
\bauthor{\bsnm{Zhao}, \binits{W.}},
\bauthor{\bsnm{Zhang}, \binits{W.}},
\bauthor{\bsnm{Ju}, \binits{L.}}:
\batitle{A numerical method and its error estimates for the decoupled
  forward-backward stochastic differential equations}.
\bjtitle{Communications in Computational Physics}
\bvolume{15}(\bissue{3}),
\bfpage{618}--\blpage{646}
(\byear{2014})
\end{barticle}
\endbibitem

\bibitem[\protect\citeauthoryear{Archibald
  et~al.}{2020}]{archibald2020efficient}
\begin{barticle}
\bauthor{\bsnm{Archibald}, \binits{R.}},
\bauthor{\bsnm{Bao}, \binits{F.}},
\bauthor{\bsnm{Yong}, \binits{J.}},
\bauthor{\bsnm{Zhou}, \binits{T.}}:
\batitle{An efficient numerical algorithm for solving data driven feedback
  control problems}.
\bjtitle{Journal of Scientific Computing}
\bvolume{85}(\bissue{2}),
\bfpage{51}
(\byear{2020})
\end{barticle}
\endbibitem

\bibitem[\protect\citeauthoryear{Archibald
  et~al.}{2023}]{archibald2023stochastic}
\begin{barticle}
\bauthor{\bsnm{Archibald}, \binits{R.}},
\bauthor{\bsnm{Bao}, \binits{F.}},
\bauthor{\bsnm{Yong}, \binits{J.}}:
\batitle{A stochastic maximum principle approach for reinforcement learning
  with parameterized environment}.
\bjtitle{Journal of Computational Physics}
\bvolume{488},
\bfpage{112238}
(\byear{2023})
\end{barticle}
\endbibitem

\bibitem[\protect\citeauthoryear{Haber and Ruthotto}{2017}]{haber2017stable}
\begin{barticle}
\bauthor{\bsnm{Haber}, \binits{E.}},
\bauthor{\bsnm{Ruthotto}, \binits{L.}}:
\batitle{Stable architectures for deep neural networks}.
\bjtitle{Inverse problems}
\bvolume{34}(\bissue{1}),
\bfpage{014004}
(\byear{2017})
\end{barticle}
\endbibitem

\bibitem[\protect\citeauthoryear{Yong and Zhou}{2012}]{yong2012stochastic}
\begin{bbook}
\bauthor{\bsnm{Yong}, \binits{J.}},
\bauthor{\bsnm{Zhou}, \binits{X.}}:
\bbtitle{Stochastic Controls: Hamiltonian Systems and HJB Equations}
vol. \bseriesno{43}.
\bpublisher{Springer},
\blocation{New York}
(\byear{2012})
\end{bbook}
\endbibitem

\bibitem[\protect\citeauthoryear{Peng}{1991}]{peng1991probabilistic}
\begin{barticle}
\bauthor{\bsnm{Peng}, \binits{S.}}:
\batitle{Probabilistic interpretation for systems of quasilinear parabolic
  partial differential equations}.
\bjtitle{Stochastics and stochastics reports (Print)}
\bvolume{37}(\bissue{1-2}),
\bfpage{61}--\blpage{74}
(\byear{1991})
\end{barticle}
\endbibitem

\bibitem[\protect\citeauthoryear{Du et~al.}{2013}]{IJNAM-10-4}
\begin{barticle}
\bauthor{\bsnm{Du}, \binits{n.}},
\bauthor{\bsnm{Shi}, \binits{J.}},
\bauthor{\bsnm{Liu}, \binits{W.}}:
\batitle{An effective gradient projection method for stochastic optimal
  control}.
\bjtitle{International Journal of Numerical Analysis and Modeling}
\bvolume{10}(\bissue{4}),
\bfpage{757}--\blpage{774}
(\byear{2013})
\end{barticle}
\endbibitem

\end{thebibliography}

\end{document}